\documentclass[11pt, a4paper,leqno]{amsart}
\usepackage{amsmath,amsthm,amscd,amssymb,amsfonts, amsbsy}
\usepackage{latexsym}
\usepackage{txfonts}
\usepackage{exscale}
\usepackage{enumerate}
\usepackage{enumitem}
\usepackage{moreenum}
 \usepackage{float}
\usepackage{mathrsfs}
\usepackage{currvita}
\usepackage[colorlinks,citecolor=red,pagebackref,hypertexnames=false]{hyperref}
\usepackage[utf8]{inputenc}

\usepackage{color}
\usepackage[active]{srcltx}

\usepackage[a4paper, total={7in, 8in}]{geometry}

\calclayout
\allowdisplaybreaks

\theoremstyle{plain}
\newtheorem{theorem}[equation]{Theorem}
\newtheorem{lemma}[equation]{Lemma}

\newtheorem{proposition}[equation]{Proposition}

\theoremstyle{definition}
\newtheorem{definition}[equation]{Definition}

\theoremstyle{remark}
\newtheorem{remark}[equation]{Remark}

\numberwithin{equation}{section}

\def\Xint#1{\mathchoice
{\XXint\displaystyle\textstyle{#1}}%
{\XXint\textstyle\scriptstyle{#1}}%
{\XXint\scriptstyle\scriptscriptstyle{#1}}%
{\XXint\scriptscriptstyle%
\scriptscriptstyle{#1}}%
\!\int}
\def\XXint#1#2#3{{\setbox0=\hbox{$#1{#2#3}{%
\int}$ }
\vcenter{\hbox{$#2#3$ }}\kern-.6\wd0}}
\def\barint{\,\Xint -} 
\def\bariint{\barint_{} \kern-.4em \barint}
\def\bariiint{\bariint_{} \kern-.4em \barint}
\renewcommand{\iint}{\int_{}\kern-.34em \int} 
\renewcommand{\iiint}{\iint_{}\kern-.34em \int} 

\newcommand{\dist}{\operatorname{dist}}

\renewcommand{\d}{\, \mathrm{d}}

\DeclareMathOperator{\supp}{supp}

\DeclareMathOperator{\diam}{diam}

\begin{document}
\allowdisplaybreaks

\title[Square Functions Controlling Smoothness and Higher-Order Rectifiability]{Square Functions Controlling Smoothness and Higher-Order Quantitative Rectifiability}

\author{John Hoffman}
\address{Department of Mathematics
	\\
	Florida State University
	\\
	Tallahassee, FL 32306, USA}
\email{jh23bq@fsu.edu}

\date{\today}

\begin{abstract}
We provide new characterizations of the $BMO$-Sobolev space $I_{\alpha}(BMO)$ for the range $0 < \alpha <2$.  
When $0 < \alpha <1$, our characterizations are in terms of square functions measuring multiscale approximation of
constants, and when $1 \leq \alpha <2$ our characterizations are in terms of square functions measuring multiscale approximation by linear functions.
\end{abstract}

\maketitle
\tableofcontents

\section{Introduction}

This paper concerns new characterizations of the $BMO$-Sobolev spaces $I_{\alpha}(BMO)$ for $0 < \alpha <2$ in terms of geometric square functions which 
measure multiscale approximation by constants for $\alpha \in (0,1)$, and by linear functions for $\alpha \in [1,2)$.  The spaces $I_{\alpha}(BMO)$ were 
introduced by Strichartz in \cite{Str} and \cite{Str2}.  They can be thought of as the image of $BMO$ under the $\alpha$\textsuperscript{th} Riesz potential, although 
to be rigorous one needs to define them as an appropriate class of distributions, which in turn coincide with functions.  These spaces occur naturally
in the theory of commutators of fractional singular integrals  (see for example \cite{CDH}, \cite{CDZ}, \cite{CHO}, \cite{CTWX}, \cite{HLY}, and \cite{M} ) and the theory of parabolic geometric measure theory (more on this below.)  The new characterization of these spaces
established in this paper provide a connection between these spaces and the theory of higher-order rectifiability.


Before stating the main Theorem of this paper, we provide a history of previous results which motivate our present inquiry.  We recall the notion of a ``$\beta$-number," which measures how well a measure
is approximated by a plane at a given point and scale.  Given a measure $\mu$ on $\mathbb R^d$ and an integer $k \in {1,...,d-1}$, we can define
the $k$-dimensional $\beta$-numbers as
\begin{align}\label{l2beta}
\beta_{2,k}(X,r) = \inf_{\text{$P$ a $k$-plane}} \left(\fint_{B_r(X)}\left(\frac{\dist(Y,P)}{r} \right)^2 \d \mu(y) \right)^{\frac{1}{2}}.
\end{align}
``$\beta$-numbers" like the one above have a rich history which is too long to discuss in full, but we will
provide some relevant context.
Such quantities were first introduced by Peter Jones in \cite{J}, where the author defined an $L^\infty$ version of the quantity above.  
 In \cite{DS1}, David and Semmes show
that, in the presence of appropriate scale-invariant regularity of the surface measure, Carleson measure estimates involving the 
$\beta$-numbers as defined in \eqref{l2beta} are equivalent to a number of quantitative conditions, including boundedness of singular integral convolution operators
and the existence of a ``corona decomposition," a multiscale decomposition of the dyadic grid of the set into trees which are
well approximated by Lipschitz graphs.  The aforementioned Carleson measure condition is
\begin{align}\label{CMest}
\int_0^R \int_{B_R(X)\cap \Sigma} \beta_{2,k}(Y,r)^2 \frac{\d\mu(Y) \d r}{r} \leq CR^k, \,\,\,\,\,\,\,\,\,\,
\text{for all $X \in \Sigma$, $R \in (0,\diam \Sigma)$}.
\end{align}
We also define, for $k \in \{1,...,d-1\}$ the \textit{upper $k$-density} and \textit{lower $k$-density}, respectively,  of a measure $\mu$ at a point $X \in \supp(\mu)$
$$\theta^{k,*}(\mu)(X) = \limsup_{r \rightarrow 0^+} \frac{\mu(B_r(X))}{(2r)^k},$$
$$\theta^{k}_*(\mu)(X) = \liminf_{r \rightarrow 0^+} \frac{\mu(B_r(X))}{(2r)^k}.$$
For $k \in \{1,...,d-1\}$, a say that a measure $\mu$ is $k$-rectifiable if there are Lipschitz functions $f_i: \mathbb R^k \rightarrow \mathbb R^d$ such that
$$\mu\bigg( \mathbb R^d \setminus \bigcup_i f_i(\mathbb R^n) \bigg) =0.$$
The definition of $C^{1,\alpha}$ rectifiability is nearly identical, the only difference being that the Lipschitz functions $f_i$ are replaced by $C^{1,\alpha}$ functions.\\

In \cite{AT} and \cite{T}, Azzam and Tolsa prove qualitative analogs of the result of David and Semmes.  Firstly, in \cite{AT}, the authors show that, 
in the presence of a mild hypothesis on the density of a measure, a qualitative version of the Carleson measure condition \eqref{CMest} implies rectifiability of the measure.  
\begin{theorem}[\cite{AT}]\label{AzzamTolsa}
Suppose $\mu$ is a measure on $\mathbb R^d$ such that the upper density $\theta^{k,*}(\mu)(X)$ of a measure
$\mu$ satisfies $0 < \theta^{k,*}(\mu)(X) < \infty$ for $\mu$-a.e. $X \in \mathbb R^d$ for some $k \in \{1,...,d-1\}$,
and one has the following qualitative version of the aforementioned Carleson measure
estimate \eqref{CMest}
$$\int_0^1 \beta_{2,k}(X,r)^2 \frac{\d r}{r} < \infty\,\,\,\, \text{$\mu$-a.e. $X \in \supp(\mu)$},$$
then $\mu$ is a $k$-rectifiable measure.
\end{theorem}
In \cite{T}, the author proves a result in the converse direction.
\begin{theorem}[\cite{T}]
Let $1 \leq p \leq 2$.  Suppose $\mu$ is a finite Borel measure in $\mathbb R^d$ which is $k$-rectifiable, $k \in \{1,...,d-1\}$.  Then
\begin{align*}
\int_0^\infty \beta_{2,k}(X,r)^2 \frac{\d r }{r} < \infty\,\,\,\, \text{ for $\mu$-a.e. $X \in \mathbb R^d$}.
\end{align*}
\end{theorem}
In \cite{ENV}, Edelen, Naber, and Valtorta prove a similar theorem to that proved by Azzam and Tolsa \cite{AT} above.  
\begin{theorem}[\cite{ENV}]
Suppose $\mu$ is a measure on $\mathbb R^d$, such that for 
some $k \in \{1,...,d-1\}$ $\theta^{k*}(\mu)(X)>0$ and $\theta^k_*(X)<\infty$  for $\mu$-a.e. $X \in \mathbb R^d$, 
and such that
$$\int_0^1 \beta_{2,k}(X,r)^2 \frac{\d r}{r} < \infty\,\,\,\, \text{$\mu$-a.e. $X \in \supp(\mu)$},$$
then $\mu$ is a $k$-rectifiable measure.
\end{theorem}

In recent years there have been several papers regarding ``$C^{1,\alpha}$ variants" of the results of Azzam and Tolsa \cite{AT} and Edelen, Naber, 
and Valtorta \cite{ENV}.  
In \cite{Ghi}, the author proves the following Theorems.
\begin{theorem}[\cite{Ghi}]\label{Ghithm1}
Let $\mu$ be a Radon measure on $\mathbb R^d$ such that
for some $k \in \{1,...,d-1\}$ $0 < \theta^k_*(\mu;X)$ and $\theta^{k*}(\mu;X) < \infty$ for $\mu$-a.e. $X \in \mathbb R^d$ and let
$\alpha \in (0,1)$.  Assume for $\mu$-a.e. $X \in \mathbb R^d$
\begin{align}\label{ghisqr1}
\int_0^1 \left(\frac{\beta_{2,k}(X,r)}{r^\alpha}\right)^2 \frac{\d r}{r} < \infty.
\end{align}
Then $\mu$ is countably $C^{1,\alpha}$-rectifiable of dimension $k$.
\end{theorem}
\begin{theorem}[\cite{Ghi}]\label{Ghithm2}
Let $\mu$ be a Radon measure on $\mathbb R^d$ such that
for some $k \in \{1,...,d-1\}$
 $0 < \theta^{k*}(\mu;x)< \infty$ for $\mu$-a.e. $X \in \mathbb R^d$
 and let
$\alpha \in (0,1)$.  Assume for $\mu$-a.e. $X \in \mathbb R^d$
\begin{align}\label{ghisqr2}
\int_0^1 \left(\frac{\beta_{2,\infty}(X,r)}{r^\alpha}\right)^2 \frac{\d r}{r} < \infty.
\end{align}
Then $\mu$ is countably $C^{1,\alpha}$-rectifiable of dimension
$k$.
\end{theorem}
We should note that in \cite{Ghi}, the author was also able to prove results pertaining to Reifenberg-flat parametrizations for sets satisfying the conditions of
the above Theorems. \\

More recently, Del Nin and Ibu \cite{DI} obtained the following Theorem.
\begin{theorem}[\cite{DI}]\label{DIthm}
Suppose $E \subseteq \mathbb R^d$ satisfies $\mathcal H^k(E) < \infty$ for some $k \in \{1,...,d-1\}$, and that for $\mathcal H^k$-a.e. $X \in E$ there exists $1 \leq p \leq \infty$ such that
\begin{align}\label{dilim}
\limsup_{r \rightarrow \infty} r^{-\alpha}\beta_p(X,r) < \infty.
\end{align}
Then $E$ is countably $C^{1,\alpha}$-rectifiable of dimension $k$.
\end{theorem}
The Theorem above is actually stated as a Corollary in \cite{DI}.  The main results of their paper pertain to sufficient conditions for $C^{1,\alpha}$-rectifiability
in terms of ``approximate tangent paraboloids."  

One should note that the square function hypotheses \eqref{ghisqr1} and \eqref{ghisqr2} of Theorems \ref{Ghithm1} and \ref{Ghithm2} imply the condition of Del Nin and Ibu \eqref{dilim}.  So, one could reasonably hypothesize that the square function hypotheses \eqref{ghisqr1} and \eqref{ghisqr2} imply a higher level of regularity than $C^{1,\alpha}$.  Before showing that the hypotheses of Theorems \ref{Ghithm1} and \ref{Ghithm2} imply something stronger than $C^{1,\alpha}$ rectifiability, it seems
necessary to understand what level of regularity quantitative versions \eqref{ghisqr1} and \eqref{ghisqr2} imply when the underlying set is a Lipschitz graph.  In our result we actually prove something more general.  The reader should keep in mind that the ``$\nu_1$'' numbers defined below are, in some sense,
equivalent to the $\beta$-numbers when the underlying function $f$ is Lipschitz.

Given a function $f:\mathbb R^d \rightarrow \mathbb R$, define
\begin{align*}
\nu_0^f(x,r) = \inf_{\text{$c$ a constant}}\left( \fint_{B_r(x)} \left(\frac{|f(y)-c|}{r}\right)^2 \d y \right)^\frac{1}{2},
\end{align*}
and
\begin{align*}
\nu_1^f(x,r) = \inf_{\text{$l$ affine}}\left( \fint_{B_r(x)} \left(\frac{|f(y)-l(y)|}{r}\right)^2 \d y \right)^\frac{1}{2}.
\end{align*}
We prove the following Theorem:
\begin{theorem}\label{mainthm}
A function $f: \mathbb R^d \rightarrow \mathbb R$ is a member of $I_{\alpha}(BMO)$ if and only if, in the case that
$0 < \alpha <1$, we have the following estimate for every $R>0$ and $z \in \mathbb R^d$
\begin{align}\label{nu0}
\int_0^R \int_{B_R(z)} \left(\frac{\nu_0^f(y,r)}{r^{\alpha-1}}\right)^2 \frac{\d y \d r}{r} \leq C R^d.
\end{align}
Additionally, in the case that the case that $1 \leq \alpha <2$, $f \in I_{\alpha}(BMO)$ if and only if $f$ satisfies the following estimate for every $R>0$ and $z \in \mathbb R^d$
\begin{align}\label{nu1}
\int_0^R \int_{B_R(z)} \left(\frac{\nu_1^f(x,r)}{r^{\alpha-1}}\right)^2 \frac{\d x \d r}{r} \leq C R^d,
\end{align}
and the following integral is finite for all $\epsilon > \alpha$
\begin{align}\label{tempg1}
\int \frac{|f(x)|}{1+|x|^{d+\epsilon}} \d x.
\end{align}
Moreover, in both cases, $\|D_\alpha f\|_*^2$, $D_\alpha f$ the $\alpha$\textit{th} fractional derivative of $f$ and $\|\cdot\|_*$ the
$BMO$ norm, is comparable to the constant $C$ in the square function estimates
\eqref{nu0} and \eqref{nu1}.
\end{theorem}
\begin{remark}
See Definition \ref{aderivdef} below for the precise definition of $D_\alpha f$.
\end{remark}
As mentioned above, when $f$ is Lipschitz function, one can show that the square function on the left hand side of \eqref{nu1} is equivalent to 
a similar square function involving the $\beta$-numbers.  More precisely, if $\Gamma = \{(x,f(x))\}_{x \in \mathbb R^{d}} \subseteq \mathbb R^{d+1}$ 
is the graph of $f$ and $\Delta_r(X)$ is used to denote the surface ball or radius $r$ of $\Gamma$ centered at $X$,
one has the bound of display \eqref{nu1} if and only if the following square function holds (perhaps with a different constant $C$ on the right hand side) for every
$Z \in \Gamma$ and $R>0$
\begin{align}\label{betaversion}
\int_0^R \int_{\Delta_R(Z)} \left( \frac{\beta_{2,d}(X,r)}{r^{\alpha-1}}\right)^2 \frac{\d \mathcal H^d\vert_{\Gamma}(X) \d r}{r} \leq CR^d.
\end{align}
When $\alpha \in (1,2)$ the display above can be seen as a quantitative version of the square functions \eqref{ghisqr1}, \eqref{ghisqr2}.  Notice that in the display above, 
$\alpha \in (1,2)$, while in displays \eqref{ghisqr1} and \eqref{ghisqr2} $\alpha \in (0,1)$, and hence the power of $r$ in the denominator of the squared quantity 
is the same in all of these square functions.

Theorem \ref{mainthm} above can be seen as an ``endpoint" version of Dorronsoro's theorem.  Dorronsoro's theorem was first proved in \cite{Do}, where the 
author characterized inhomogeneous Sobolev spaces using a Littlewood-Paley decomposition in terms of coefficients like the $\nu_0$ and $\nu_1$ above.  
We should note that our 
coefficients are defined slightly differently to make the connection to the $\beta$ numbers more explicit.  Over the years different methods of proof
of Dorronsoro's theorem have appeared.  In \cite{A}, the author proves an estimate similar to the one proved in \cite{Do}, but in the case that $\alpha=1$ and in terms of homogeneous Sobolev spaces.  In \cite{O}, the author proves Dorronsoro's theorem in the case $\alpha =1$ using geometric methods.  Moreover, in \cite{O}, the 
author proves a parabolic version of Dorronsoro's theorem.  In \cite{HN}, the authors prove Dorronsorro's theorem in the case that $\alpha$ is in integer
using analytic methods.  We note that Proposition \ref{dorronsoro} of this paper is proved using a modification of Azzam's argument given in \cite{A}.

BMO-Sobolev spaces play a central role in the theory of \textit{parabolic uniform rectifiability}.  Let us refer to Lip$(1,1/2)$ graphs (i.e. graphs given by functions which are Lipschitz with respect to the parabolic norm)
\textit{parabolic Lipschitz} graphs, and Lip$(1,1/2)$ graphs with underlying $(1/2)$-time derivative in parabolic BMO \textit{regular parabolic Lipschitz} graphs.  We refer the reader to a recent paper of Dindo\v{s} and S\"atterqvist \cite{DSa} where the authors show that \textit{regular parabolic Lipschitz} functions can be characterized through a parabolic version of the BMO-Sobolev norm developed by Strichartz in \cite{Str}.  In the theory of parabolic uniform rectifiability, \textit{parabolic regular Lipschitz} graphs play a role analogous to that of Lipschitz graphs in the
standard theory of uniform rectifiability.  

In addition to proving Theorem \ref{mainthm} above, this paper provides proofs for some fundamental facts about $BMO$-Sobolev spaces which the author could not find in the existing literature.  In particular, we provide proofs that distributions in $I_{\alpha}(BMO)$ can be identified with functions when $0<\alpha<2$, prove that the initial distributional 
definition of $I_{\alpha}(BMO)$ is well-defined,
and prove that the class of distributions
which $I_{\alpha}(BMO)$ naturally live in 
are the dual space of Schwartz functions with 
vanishing moments of high enough degree.  One
can find all of these proofs in the Appendix.
We should note that \cite{Str} gives a very brief
description of a slightly different way to construct the 
space $I_{\alpha}(BMO)$, but does not provide
any proofs.

\section{Notation, Definitions}

We will consider functions having domain $\mathbb R^d$, endowed with the standard metric.  Given $r>0$ and $x \in \mathbb R^d$, we let 
$$B_r(x) = \{y \in \mathbb R^d: |x-y|<r\}.$$

When denoting function spaces, we will omit any mention of $\mathbb R^d$.  For example, we will use the notation $L^2$, $BMO$, $H^1$ to denote the spaces
$L^2(\mathbb R^d)$, $BMO(\mathbb R^d)$ and $H^1(\mathbb R^d)$ (here $H^1$ is the Hardy space.)

The letters $C$ and $c$ to denote harmless constants
which depend only on the structural parameters of a given hypothesis.  We will use the notation $a \lesssim b$ to mean $a \leq Cb$, $a \gtrsim b$
to mean $a \geq Cb$, and $a \approx b$ to mean that both $a \lesssim b$ and $a \gtrsim b$ hold.  Given a constant $D>0$, we will 
often we will adopt the notation $a \lesssim_D b$, $a \leq CD b$,
for some constant $C$ which is independent of $D$.  $a \approx_D b$ and $a \gtrsim_D b$ are defined analogously.

Of particular interest to us is the space $BMO$.

\begin{definition}
The space $BMO$ is the set of all equivalences classes of functions $f: \mathbb R^d \rightarrow \mathbb R$ endowed with the equivalence
relation which says that $f \sim g$ if $f-g$ is a constant and such that
$$\|f\|_{*} := \sup_{x \in \mathbb R^d, r>0} \fint_{B_r(x)} |f(y) - \fint_{B_r(x)} f(z) \d z| \d y<\infty.$$
We note that the quantity above is unchanged if
one adds a constant to $f$.  We will often use the 
notation
$$\langle f \rangle_{B_r(x)} = \fint_{B_r(z)} f(z) \d z.$$
\end{definition}
We want to define the space $I_{\alpha}(BMO)$.  First, we need to define the $\alpha$\textsuperscript{th} fractional integral of a Schwartz function.  We 
adopt the notation $\mathcal S$ for the Schwartz space on $\mathbb R^d$.
\begin{definition}
Given $\varphi\in \mathcal S$ and $0<\alpha < d$ we define
$$I_{\alpha}(\varphi)(x) = c_{d,\alpha} \int_{\mathbb R^d} \frac{\varphi(y)}{|x-y|^{d-\alpha}} \d y,$$
where $c_{n,\alpha} = 2^{-\alpha} \pi^{-d/2} \Gamma((d-\alpha)/2) \Gamma(\alpha/2)^{-1}$.  Alternatively, for any $\alpha>0$, we can define $I_{\alpha}(\varphi)$ on the
Fourier transform side as
$$I_{\alpha}(\varphi)(x) = \big((2\pi |\xi|)^{-\alpha} \hat \varphi(\xi) \check{\big)}(x),$$
When $0<\alpha<d$, these definitions coincide.
\end{definition}

\begin{definition}\label{aderivdef}
Given $\alpha>0$, we define the $\alpha$\textsuperscript{th} fractional derivative of a function $\varphi \in \mathcal S$ via the multiplier $(2\pi)^\alpha|\xi|^\alpha$, i.e.
$$(D_\alpha \varphi)(x) := (2\pi)^\alpha(|\xi|^\alpha \hat \varphi(\xi)\check)(x).$$
\end{definition}

Observe that for $\varphi \in \mathcal S$ and any $\alpha>0$, $I_{\alpha} D_{\alpha}\varphi = D_{\alpha} I_{\alpha} \varphi = \varphi$.

We let $\mathcal S'$ denote the space of tempered distributions.  Given an integer $l\geq0$ let $P_l$ denote the polynomials of degree at most $l$, and
define $\mathcal S'/P_l$ to be the set of ``tempered distributions modulo polynomials of degree $l$."  By this, we mean the $\mathcal S'$
endowed with the equivalence relation that says $f,g \in \mathcal S'$ are equivalent if $f-g \in P_l$. 

We now define the functions which have dual space $\mathcal S'/P_l$.  
$\mathcal S_{l}$ is the subset $\varphi \in \mathcal S$ such that, for every multi-index $\gamma$ with $|\gamma| \leq l$, one has
$$\int x^\gamma \varphi(x) \d x =0.$$
We endow $\mathcal S_l$ with the topology inherited from $\mathcal S$.

We state the following Proposition, which is proved in the appendix of this paper.

\begin{proposition}\label{Sldual}
The dual space of $\mathcal S_l$ can be identified with $\mathcal S' / P_l$.
\end{proposition}

We are now ready to define the space $I_{\alpha}(BMO)$.

\begin{definition}\label{ialphabmodef}
Given $\alpha>0$, $I_{\alpha}(BMO)$ is the subset of $\mathcal S'/P_{\lfloor \alpha \rfloor}$ consisting of distributions $I_{\alpha}(b)$, $b \in BMO$ 
defined via the following action on $\varphi \in \mathcal S_{\lfloor \alpha \rfloor}$
$$I_{\alpha}(b)(\varphi) := \int b(x)I_{\alpha}(\varphi)(x) \d x.$$
\end{definition}
\begin{remark}
In the appendix we prove that the integral in the definition of $I_{\alpha}(b)$ above is well-defined.
\end{remark}

As mentioned above, the spaces $I_{\alpha}(BMO)$ coincide with functions modulo polynomials (constants in the case that $0<\alpha<1$
and linear functions when $1 \leq \alpha <2$.) We will often abuse language and simply refer to ``functions in $I_{\alpha}(BMO)$."  The following Proposition is proved in the appendix.

\begin{proposition}\label{distcoincide}
For $0<\alpha<1$, members of $I_{\alpha}(BMO)$ coincide with functions modulo constants.  In other words, given a tempered distribution modulo constants $f \in I_{\alpha}(BMO)$, we can identify $f$ with a function modulo
constants.  If $1 \leq \alpha <2$, members of $I_{\alpha}(BMO)$ coincide with functions modulo linear polynomials.
\end{proposition}

Given $f :\mathbb R^d \rightarrow
\mathbb R$ will adopt the following
notation for the $C^\alpha$ norm of $f$, 
$0<\alpha<1$
$$[f]_\alpha = \sup_{x \neq y}
\frac{|f(x)-f(y)|}{|x-y|^\alpha}.$$

We will need the following Proposition about the continuity properties
of $I_{\alpha}(BMO)$ for $0<\alpha<2$, which is proved in the paper of Strichartz \cite{Str}.
\begin{proposition}\label{iaimpholder}
For $0<\alpha<1$, (equivalence classes of) functions in $I_{\alpha}(BMO)$ are H\"older continuous with exponent $\alpha$.  For $1 < \alpha < 2$, (equivalence classes of)
functions in $I_{\alpha}(BMO)$ are differentiable with $(\alpha-1)$-H\"older continuous gradients.  If $\alpha = 1$, (equivalence classes of) functions in $I_{\alpha}(BMO)$
are differentiable almost everywhere with gradients in $BMO$.  In particular, $[f]_{\alpha} \lesssim \|D_\alpha f\|_*$ for $\alpha \in (0,1)$, and $[\nabla f]_{\alpha-1} \lesssim
\|D_\alpha f\|_*$ for $1 < \alpha <2$.
\end{proposition}

\section{$I_\alpha(BMO)$ Implies Square Function Bounds}
In this section we show that functions in $I_{\alpha}(BMO)$ necessarily satisfy the square function bounds \eqref{nu0} and \eqref{nu1} for $0 < \alpha <1$ and $1 \leq \alpha <2$, 
respectively.  We begin by showing that localizations of $I_{\alpha}(BMO)$ functions satisfy appropriate bounds which comport with the scale of the localizations.  As we shall see,
the case when $\alpha = 1$ is fundamentally different.  We should note that Propositions 
\ref{smoothsio1} and \ref{smoothsio2} are somewhat
similar to a Lemma proved in the appendix of \cite{H2}, although our case is somewhat more difficult because we need to analyze the singular integral
which makes an appearance in the final term of 
our estimate in the case that $1 \leq \alpha < 2$.

\begin{proposition}\label{smoothsio1}
Suppose $f \in C^\infty$, $1 < \alpha <2$.  Let $\phi:\mathbb R^d \rightarrow \mathbb R$ be smooth, equal to $1$ on $B_2(0)$, equal to $0$
outside of $B_4(0)$, and such that $0 \leq \phi \leq 1$ everywhere.  Given $z \in \mathbb R^d$, $R>0$, define
\begin{align*}
g(x):=\phi\left(\frac{x-z}{R}\right)\big(f(x)-f(z) - \nabla f(z) \cdot(x-z)\big).
\end{align*}
Then we have the following bound
\begin{align}
\int_{\mathbb R^d} \big||\xi|^\alpha \hat g(\xi) \big|^2 \d \xi \leq C R^d.
\end{align}
Where $C = C(d,\alpha)\|D_{\alpha}f\|_*^2$.
\end{proposition}
\begin{proof}
Recall that, for $\alpha \in (1,2)$,  $[\nabla f]_{\alpha-1} \lesssim \| D_{\alpha}\|_*$.  Over the course of this proof, we will bound terms on the order of $[\nabla f]_{\alpha-1}^2R^d$, which is dominated by 
$\|D_\alpha f\|_*^2 R^d$.

We have
\begin{align*}
 \int_{\mathbb R^d} \left(\hat g (\xi)|\xi|^{\alpha}\right)^2 \d \xi \approx \int_{\mathbb R^d} \left( p.v. \int_{\mathbb R^d} \frac{g(x)-g(y)}{|x-y|^{d+\alpha}} \d y \right)^2 \d x .
\end{align*}
The display above uses standard facts about integral representation of the fractional Laplacian for $\alpha$ in the range $0 < \alpha <2$.
We bound expression on the right hand side above up to 
a harmless constant as follows
\begin{align*}
\int_{|x-z| \geq 100 R} \left(p.v. \int_{\mathbb R^d} \frac{g(x)-g(y)}{|x-y|^{d+\alpha}} \d y\right)^2 \d x 
+ \int_{|x-z| <100 R} \left(p.v. \int_{\mathbb R^d} \frac{g(x)-g(y)}{|x-y|^{d+\alpha}} \d y\right)^2 \d x = I + II.
\end{align*}
We estimate term $I$ using the fact that $g(x)=0$ for all $x$ in appearing in the inner integral, the fact that $|f(y)-f(z) - \nabla f(z) \cdot (y-z)|
\lesssim_{[\nabla f]_{\alpha-1}} R^{\alpha}$, and the fact that $|x-y| \gtrsim |x-z|$ for all $x,y,z$ in the inner integral for which the integrand is non-zero.
\begin{multline}\label{termIest}
I = \int_{|x-z| \geq 100 R} \left( p.v. \int_{\mathbb R^d} \frac{g(y)}{|x-y|^{d+\alpha}} \d y \right)^2 \d x\\
= \int_{|x-z| \geq 100 R} \left( p.v. \int_{\mathbb R^d} \frac{\phi\left(\frac{y-z}{R}\right)\big(f(y)-f(z)-\nabla f(z) \cdot (y-z)\big)}{|x-y|^{d+\alpha}} \d y \right)^2 \d x\\
\lesssim_{[\nabla f]_{\alpha-1}^2} R^{2\alpha}\int_{|x-z| \geq 100R} \left( \int_{|y-z| \leq 4R} \frac{1}{|x-y|^{d+\alpha}} \d y \right)^2 \d x\\
\lesssim R^{2d+2\alpha} \int_{|x-z| \geq 100R} \frac{1}{|x-z|^{2d+2\alpha}} \d x \approx R^d.
\end{multline}
So, $I \leq C[\nabla f]_{\alpha-1}^2 R^d$.  \\

Now we turn our attention to term $II$.  We write
\begin{multline*}
g(x) -g(y) = \left(\phi\left(\frac{x-z}{R} \right)- \phi\left(\frac{y-z}{R} \right) \right)\left(f(x)-f(z) - \nabla f(z) \cdot (x-z)\right) 
\\ - \phi\left(\frac{y-z}{R}\right)\left(f(y)-f(x) - \nabla f(z) \cdot (y-x)\right).
\end{multline*}
We use this decomposition to split term II
\begin{multline*}
II \lesssim \int_{|x-z| < 100 R} \left(p.v. \int_{|x-y| \leq 10 R}\frac{f(x)-f(z)- \nabla f(z) \cdot (x-z)}{|x-y|^{d+\alpha}} 
\left(\phi\left(\frac{x-z}{R} \right)- \phi\left(\frac{y-z}{R} \right) \right)\d y \right)^2 \d x\\
+\int_{|x-z| <100 R} \left( \int_{|x-y| \geq 10 R}\frac{f(x)-f(z)- \nabla f(z) \cdot (x-z)}{|x-y|^{d+\alpha}} 
\left(\phi\left(\frac{x-z}{R} \right)- \phi\left(\frac{y-z}{R} \right) \right)\d y \right)^2 \d x\\
+\int_{|x-z| <100 R} \left(p.v. \int_{\mathbb R^d}\frac{f(y)-f(x)-\nabla f(z)\cdot (y-x)}{|x-y|^{d+\alpha}}\phi\left(\frac{y-z}{R}\right) \d y \right)^2 \d x 
\\
=: III + IV+ V.
\end{multline*}
We first estimate term $III$.  We first record the estimates
$$|f(x)-f(z) -\nabla f(z) \cdot (x-z)| \lesssim_{[\nabla f]_{\alpha-1}} R^{\alpha},$$
$$\left|\phi\left(\frac{y-z}{R} \right) - \phi\left(\frac{x-z}{R} \right)-   \nabla \left(\phi\left(\frac{x-z}{R} \right)\right) 
\cdot (y-x)\right|\lesssim \frac{|x-y|^2}{R^2}.$$
and write term $III$ as (using that the added term involving the gradient will integrate to $0$)
\begin{multline*}\label{psitaylor}
\int_{|x-z| <100 R} \bigg(p.v. \int_{|x-y| \leq 10R}\frac{f(x)-f(z) - \nabla f(z) \cdot (x-z)}{|x-y|^{d+\alpha}}\\
 \cdot \left(\phi\left(\frac{x-z}{R} \right) + \nabla \left(\phi\left(\frac{x-z}{R} \right)\right) 
\cdot (y-x) - \phi\left(\frac{y-z}{R} \right) \right)\d y \bigg)^2 \d x.
\end{multline*}
Using the estimates we just mentioned above, we see that term $III$ is bounded on the order of
\begin{align*}
[\nabla f]_{\alpha-1}^2\int_{|x-z| < 100R} R^{2\alpha-4} \left(\int_{|x-y| \leq 10R} \frac{1}{|x-y|^{d+\alpha-2}} \d y \right)^2 \d x \lesssim [\nabla f]_{\alpha-1}^2R^d.
\end{align*}
For term $IV$ we use the estimate on $|f(x)-f(z)-\nabla f(z) \cdot (x-z)|$ above, along with
$$\left| \phi\left(\frac{x-z}{R} \right)- \phi\left(\frac{y-z}{R} \right)\right| \lesssim \frac{|x-y|}{R}.$$
So, we have
$$IV \lesssim_{[\nabla f]_{\alpha-1}^2} \int_{|x-z| < 100R} R^{2\alpha-2} \left(\int_{|x-y| \geq 10R} \frac{1}{|x-y|^{d+\alpha-1}} \d y \right)^2 \d x \lesssim R^d.$$

We now decompose term $V$
\begin{multline*}
V \lesssim \int_{|x-z| < 100R} \left( \int_{|x-y| \geq 10 R} \frac{(\nabla f(z)-\nabla f(x)) \cdot (x-y)}{|x-y|^{d + \alpha}} \phi\left(\frac{y-z}{R}\right)
\d y \right)^2 \d x \\
+ \int_{|x-z| < 100R} \left( p.v. \int_{|x-y| \leq 10R} \frac{(\nabla f(z)-\nabla f(x)) \cdot (x-y)}{|x-y|^{d+\alpha}}\phi\left(\frac{y-z}{R}\right)\d y \right)^2 \d x \\
+\int_{|x-z| < 100R} \left( p.v. \int_{\mathbb R^d} \frac{f(y)-f(x) - \nabla f(x) \cdot(y-x)}{|x-y|^{d + \alpha}} 
\phi\left(\frac{y-z}{R}\right) \d y \right)^2 \d x \\
=: VI + VII + VIII.
\end{multline*}
For term $VI$ we use the estimate 
$|\nabla f(z)-\nabla f(x)| \lesssim_{[\nabla f]_{\alpha-1}} |z-x|^{\alpha-1}\lesssim R^{\alpha-1}$ for any $x$ and $z$ in the integrand.  So
\begin{multline*}
VI = \int_{|x-z| < 100R} \left( \int_{|x-y| \geq 10 R} \frac{(\nabla f(z)-\nabla f(x)) \cdot (x-y)}{|x-y|^{d + \alpha}} 
\phi\left(\frac{y-z}{R}\right) \d y \right)^2 \d x \\
\lesssim_{[\nabla f]_{\alpha-1}^2} \int_{|x-z| < 100R} \left(\int_{|x-y| \geq 10R} \frac{R^{\alpha-1}}{|x-y|^{d+\alpha-1}} \phi\left(\frac{y-z}{R}\right)  \d y \right)^2 \d x\\
\approx_{[\nabla f]_{\alpha-1}^2} R^{2\alpha-2} \int_{|x-z| <100R} \left( \int_{CR}^\infty r^{-\alpha} \d r \right)^2 \d x \lesssim R^d.
\end{multline*}
And hence $VI \lesssim [\nabla f]_{\alpha-1}^2R^d$.  To estimate $VII$, we note that 
$$\frac{(\nabla f(z)-\nabla f(x)) \cdot (x-y)}{|x-y|^{d+\alpha}} 1_{ _{|x-y| \leq 10R}}.$$
has mean value zero when considered as a kernel in $y$, and hence we can write this term as
$$\int_{|x-z| < 100R} \left( p.v. \int_{|x-y| \leq 10R} \frac{(\nabla f(z)-\nabla f(x)) \cdot (x-y)}{|x-y|^{d+\alpha}}\left(\phi\left(\frac{y-z}{R}\right)
- \phi\left(\frac{x-z}{R} \right)\right)\d y \right)^2 \d x .$$
The reader will now be familiar with the techniques used to show that $VII \leq C[\nabla f]^2_{\alpha-1}R^d$.  \\

We now come to term $VIII$.  So far, we have not used that $f \in I_{\alpha}(BMO)$, but only that
$\nabla f$ satisfies H\"older estimates.  However, this term will require us to use the full force of the hypotheses of the Proposition.  Consider the kernel $K$
$$K(x,y) = \frac{f(y)-f(x) - \nabla f(x) \cdot(y-x)}{|x-y|^{d + \alpha}} .$$
We claim that $K(x,y)$ satisfies the standard Calder\'on-Zygmund kernel estimates.  We will verify
\begin{align*}
&\text{(size)}&|K(x,y)| &\lesssim \frac{1}{|x-y|^{d}},\,\, \text{$x\neq y$},\\
&\text{(smothnesss)}&|K(x+h,y) - K(x,y)| + |K(x,y+h)-K(x,y)| &\lesssim \frac{|h|}{|x-y|^{d+1}},\,\, 2|h| \leq |x-y|.
\end{align*}
We will need both of these estimates to apply the $T(1)$ Theorem below.  For 
fixed $x,y$, $x \neq y$, and $h$ such that $2|h| \leq |x-y|$, we have
\begin{multline*}
|K(x,y+h) - K(x,y)| 
\leq  \left|\frac{\big(f(y+h)-f(x) - \nabla f(x)\cdot(y+h-x)\big)}{|x-y-h|^{d+\alpha}} 
- \frac{\big(f(y)-f(x) - \nabla f(x) \cdot (y-x)\big)}{|x-y-h|^{d+\alpha}} \right| \\
+ \left| \left(\frac{1}{|x-y-h|^{d+\alpha}} - \frac{1}{|x-y|^{d+ \alpha}} \right)
 \big(f(y)-f(x) - \nabla f(x) \cdot(y-x) \big) \right| \\
\leq \frac{|f(y+h)-f(y)-\nabla f(y) \cdot h|}{|x-y-h|^{d+\alpha}} + \frac{|(\nabla f(y)-\nabla f(x)) \cdot h|}{|x-y-h|^{d+\alpha}}\\
+ \left| \left(\frac{1}{|x-y-h|^{d+\alpha}} - \frac{1}{|x-y|^{d + \alpha}} \right)
 \big(f(y)-f(x) - \nabla f(x) \cdot(y-x\big) \right|.
\end{multline*}
Using standard estimates, we see that that the last three terms above are bounded on the order of $|h|/|x-y|^{d+1}$.  Similar reasoning yields that
$|K(x+h,y) - K(x,y)| \lesssim |h|/|x-y|^{d+1}$.  The Calder\'on-Zygmund size estimate follows from the estimate
$|A(y)-A(x) - \nabla A(x) \cdot(y-x)| \leq C |x-y|^\alpha$.  Notice the constants in both the Calder\'on-Zygmund size and smoothness are harmless
constants times $[\nabla f]_{\alpha-1}$.  

We now define the the operator
$$T[h](x) = p.v. \int_{\mathbb R^d} K(x,y)h(y) \d y.$$
We need to show that $T:L^2 \rightarrow L^2$.  First, let us focus our attention on the adjoint operator $T^*$, which is given by the kernel
$$K^*(x,y) = -\frac{f(x)-f(y) -\nabla f(y) \cdot (x-y)}{|x-y|^{d+\alpha}}.$$
We are going to verify the hypotheses of the $T(1)$ theorem: that $T(1), T^*(1) \in BMO$ and that $T^*$ satisfies the weak boundedness property.
Notice that, by our hypotheses,
$$T[1](x) = p.v. \int_{\mathbb R^d} K(x,y) \d y \in BMO, \,\,\, \|T[1]\|_{*} = \|D_\alpha f\|_*.$$
To see that $T^*[1] \in BMO$ with $\|T^*(1)\|_*=\|D_\alpha f\|_*$, notice that $T^*[1]=-T[1]$.
So, we just need to check that
$$p.v. \int \frac{x-y}{|x-y|^{d+\alpha}} \nabla f(y) \d y \in BMO,$$
with an appropriate bound on the $BMO$ norms.  But it is clear that the above expression is a constant times $D_{\alpha}f$, at least if one first checks the distributional identity
$$\bigg(p.v. \frac{x_j-y_j}{|x-y|^{d+\alpha}} \hat{\bigg)} (\xi) = iC(d,\alpha) \xi_j |\xi|^{\alpha-2}.$$

We still need to check the weak boundedness of $T^*$, i.e. that for all balls $B_M(z_0)$ and $\varphi_1,\varphi_2 \in C_0^\infty(B_M(z_0))$
\begin{equation}\label{wbpconstant}
|\langle \varphi_1, T^* \varphi_2 \rangle| \leq Cr^n(\|\varphi_1\|_\infty + r\|\nabla \varphi_1\|_\infty)(\|\varphi_1\|_\infty + r\|\nabla \varphi_2\|_\infty)
\end{equation}
for some constant $C>0$.  This was essentially proved in Lemma 4.3 of \cite{H}.  Here the author proves
that the same operator, but with $\alpha=1$, satisfies the weak boundedness property if $\nabla f \in BMO$.  To be precise, the author shows that inequality
\eqref{wbpconstant} holds with $C \approx \|\nabla f\|_*$.  In our setting, the exact same argument
works, the only change required being that instead of obtaining bounds in terms of $\|\nabla f\|_*$, we obtain bounds in terms of $[\nabla f]_{\alpha-1}$ by
using Campanato space estimates.  In particular, $T^*$ satisfies the weak boundedness property above with $C \approx [\nabla f]_{\alpha-1}$.  So, by the $T(1)$ Theorem, $T^*:L^2(\mathbb R^d) \rightarrow L^2(\mathbb R^d)$ with $\|T^*\|_{L^2 \rightarrow L^2} \leq C([\nabla f]_{\alpha-1} + \|D_{\alpha}f\|_*)$.  This mean that $T$ is $L^2 \rightarrow L^2$ bounded with the same constant, so
$$VIII \leq C ([\nabla f]_{\alpha-1}^2 + \|D_\alpha f\|_*^2)R^d.$$
We conclude the proof by recalling that $[\nabla f]_{\alpha-1} \lesssim \|D_{\alpha} f\|_*$. 
\end{proof}

\begin{proposition}\label{smoothsio2}
Suppose $f \in C^\infty$, $0 < \alpha <1$.  Let $\phi:\mathbb R^d \rightarrow \mathbb R$ be smooth, equal to $1$ on $B_2(0)$, equal to $0$
outside of $B_4(0)$, and such that $0 \leq \phi \leq 1$ everywhere.  Given $z \in \mathbb R^d$, $R>0$, define
\begin{align*}
g(x):=\phi\left(\frac{x-z}{R}\right)\big(f(x)-f(z)\big).
\end{align*}
Then we have the following bound
\begin{align}
\int_{\mathbb R^d} \big||\xi|^\alpha \hat g(\xi) \big|^2 \d \xi \leq CR^d,
\end{align}
where $C \approx \|D_{\alpha}f\|_*^2.$
\end{proposition}
\begin{proof}
The proof of this Proposition is extremely similar to the proof of Proposition \ref{smoothsio1}.  It only differs in minor ways, and these differences actually make the proof even easier.  We provide an outline for the reader.  We remark that, whenever bounds were obtained for terms on the order of $[\nabla f]_{\alpha-1}^2$
 in the proof of the last Proposition, the corresponding bounds in the proof of this proposition will be on the order of $[f]_{\alpha}^2.$

We perform a similar localization
$$g(x) = \phi\left( \frac{x-z}{R} \right) (f(x)-f(z)).$$
As in the previous proposition, we have
$$\int_{\mathbb R^d} \big||\xi|^\alpha \hat g(\xi) \big|^2 \d \xi  \approx \int_{\mathbb R^d} \left( p.v. \int_{\mathbb R^d} \frac{g(x)-g(y)}{|x-y|^{d+\alpha}} \d y \right)^2 \d x.$$
We bound this from above by a constant times $I + II$,
where
$$I =  \int_{|x-z| \geq 100 R} \left( p.v. \int_{\mathbb R^d} \frac{g(x)-g(y)}{|x-y|^{d+\alpha}} \d y \right)^2 \d x,$$
$$II = \int_{|x-z| \leq 100 R} \left( p.v. \int_{\mathbb R^d} \frac{g(x)-g(y)}{|x-y|^{d+\alpha}} \d y \right)^2 \d x.$$
 The estimate on term $I$ is almost exactly the same, except in the second line of \eqref{termIest}
we instead have the integral
$$ \int_{|x-z| \geq 100 R} \left( p.v. \int_{\mathbb R^d} \frac{\phi\left(\frac{y-z}{R}\right)\big(f(y)-f(z)\big)}{|x-y|^{d+\alpha}} \d y \right)^2 \d x.$$
Term $II$ is decomposed into terms $III$, $IV$, and $V$ which are exactly the same as in the previous case, only without the term $\nabla f(z) \cdot (x-z)$
in the numerator for term $III$ and $V$,  and with term $V$ being
being
$$\int_{|x-z| <100 R} \left(p.v. \int_{\mathbb R^d}\frac{f(x)-f(y)}{|x-y|^{d+\alpha}}\phi\left(\frac{y-z}{R}\right) \d y \right)^2 \d x.$$
Term $III$ is then estimated using $|f(z)-f(x)| \lesssim R^{\alpha}$ and 
$$\left| \phi\left(\frac{x-z}{R} \right)- \phi\left(\frac{y-z}{R} \right) \right| \lesssim \frac{|x-y|}{R}.$$
Term $IV$ is estimated in the same way as the case when $1 < \alpha <2$.  Now we estimate term $V$ in the way that term $VIII$ was estimated above,
with 
$$K(x,y) = \frac{f(x)-f(y)}{|x-y|^{d+\alpha}}.$$
It is routine (and very similar to the case above) to check that the fact that $f \in C^\alpha$ implies that $K(x,y)$ satisfies the standard Calder\'on-Zygmund  kernel estimates.
Notice that, in this case, we can apply the $T(1)$ theorem directly to the singular integral associated to $K(x,y)$.  In particular, the weak boundedness property is trivial to verify
because the kernel is antisymmetric.
\end{proof}

Let $\psi$ be a smooth, radial, compactly supported, even, positive function with $\int \psi =1$.  
We set $\psi_r(x) = r^{-d}\psi(x/r)$ (and extend this notation to apply to functions other than $\psi$ in the sequel.)  
Given a function $f:\mathbb R^d \rightarrow \mathbb R$, set
\begin{align}\label{overlinenu}
\overline{\nu}^f_1(x,r) &:=\left( \fint_{B_r(x)} \left(\frac{|f(y) - (f\ast \psi_r)(x) -  \nabla (f \ast \psi_r)(x) \cdot (y-x)|}{r} \right)^2 \d y \right)^{\frac{1}{2}},\\
\overline{\nu}^f_0(x,r) &:=\left( \fint_{B_r(x)} \left(\frac{|f(y) - (f\ast \psi_r)(x)|}{r} \right)^2 \d y \right)^{\frac{1}{2}}.
\end{align}

\begin{proposition}\label{dorronsoro}
Suppose that $f \in L^2$ and $0 < \alpha <2$.   We have
\begin{align*}
\int_0^\infty \int_{\mathbb R^d} \left(\frac{\overline{\nu}_{\lfloor \alpha \rfloor}^f(x,r)}{r^{\alpha-1}}\right)^2  \frac{\d x \d r}{r} \lesssim \int_{\mathbb R^d} |\hat f(\xi) |\xi|^\alpha|^2 \d \xi.
\end{align*}
\end{proposition}

\begin{proof}

We first consider the case when $0 < \alpha < 1$.  We have
\begin{multline}\label{FTbd}
\int_0^\infty \int_{\mathbb R^d} \left(\frac{\overline \nu_{0}^f(x,r)}{r^{\alpha-1}}\right)^2  \frac{\d x \d r}{r}  \\
= \int_0^\infty \int_{\mathbb R^d}  \int_{B_r(x)}|f(y)- f \ast \psi_r(x)|^2 \d y  \d x \frac{\d r}{r^{d+1+2\alpha}} \\
=  \int_0^\infty \int_{B_r(0)} \int_{\mathbb R^d} |f(x-y) - f \ast \psi_r(x)|^2 \d x \d y \frac{\d r}{r^{d+1+2\alpha}}  \\
=  \int_0^\infty \int_{B_r(0)} \int_{\mathbb R^d} |\hat{f}(\xi)e^{-2\pi i y \cdot \xi} - \hat{f}(\xi) \hat \psi(r\xi)|^2 \d \xi 
\d y \frac{\d r}{r^{d + 1 +2\alpha}}  \\
= \int_{\mathbb R^d}|\hat f(\xi)|^2 \int_0^\infty \int_{B_r(0)} |e^{-2 \pi i y\cdot \xi} - \hat \psi(r \xi)|^2 \d y \frac{\d r}{r^{d+1 +2\alpha}} \d \xi.
\end{multline}
We change variables in $r$, and then in $y$ to get that the expression above is 
\begin{multline*}
\int_{\mathbb R^d} |\hat f(\xi)|^2|\xi|^{d+2\alpha} \int_0^\infty \int_{B_{\frac{r}{|\xi|}(0)}} |e^{-2 \pi i y\cdot \xi} - \hat \psi(r)|^2 \d y \frac{\d r}{r^{d+2\alpha}} \d \xi \\
\approx \int_{\mathbb R^d}|\hat f(\xi)|\xi|^{\alpha} |^2 \int_0^\infty \int_{B_{r}(0)} 
|e^{-2 \pi i y\cdot \frac{\xi}{|\xi|}} - \hat \psi(r)|^2 \d y \frac{\d r}{r^{d+1 +2\alpha}} \d \xi .
\end{multline*}
Notice here that we are abusing notation by writing $\hat \psi(r)$.  This is justified because $\psi$, and hence also $\hat \psi$, is radial.\\

We are going to show that the inner integrand $|e^{-2\pi i y \cdot \frac{\xi}{|\xi|}} -\hat \psi (r)|$ is bounded on the order of 
$\min \{r,1\}$ when $y \in B_r(0)$.  First we show that the integrand is less than $cr$ for $0 \leq r \leq  1$.  We use the triangle inequality
\begin{align*}
|e^{-2\pi i y \cdot \frac{\xi}{|\xi|}} -\hat \psi (r)| \leq \left|\sin\left(2\pi y \cdot \frac{\xi}{|\xi|}\right)\right|
+\left|\cos\left(2\pi y \cdot \frac{\xi}{|\xi|}\right) - \hat\psi\left(2\pi y \cdot \frac{\xi}{|\xi|}\right)\right| 
+\left|\hat\psi\left(2\pi y \cdot \frac{\xi}{|\xi|}\right) - \hat\psi(r) \right|\\
=: I +II +III.
\end{align*}
Clearly we have $I,III \lesssim r$.  We need to bound $II$, but this is easy because (by the properties of $\psi$ listed above) the first-order Taylor polynomial
of $\cos(\cdot)$ at $0$ is the same as the first-order Taylor polynomial of $\hat \psi(r)$ at $0$ (viewing $\hat \psi$ as a function from $\mathbb R \rightarrow \mathbb R$.)  The fact that the integrand is bounded on the order of $1$ when $r\geq 1$ is trivial.

Our bounds on the inner integrand allow us to perform the following estimate
\begin{multline*}
\int_0^\infty \int_{B_{r}(0)} 
|e^{-2 \pi i y\cdot \frac{\xi}{|\xi|}} - \hat \psi(r)|^2 \d y \frac{\d r}{r^{d+1 +2\alpha}}\\
\lesssim \int_0^1 \int_{B_{r}(0)} 
|e^{-2 \pi i y\cdot \frac{\xi}{|\xi|}} - \hat \psi(r)|^2 \d y \frac{\d r}{r^{d+1 +2\alpha}} + 
\int_1^\infty \int_{B_{r}(0)} 
|e^{-2 \pi i y\cdot \frac{\xi}{|\xi|}} - \hat \psi(r)|^2 \d y \frac{\d r}{r^{d+1 +2\alpha}} \\
\lesssim \int_0^1 r^{1-2\alpha} dr + \int_1^\infty \frac{\d r}{r^{1+2\alpha}} \lesssim 1.
\end{multline*}
So, bringing everything together, we have
\begin{align*}
\int_0^\infty \int_{\mathbb R^d} \left(\frac{\nu_{0}^f(x,r)}{r^{\alpha-1}}\right)^2  \frac{\d x \d r}{r}\lesssim \int_{\mathbb R^d} |\hat f(\xi) |\xi|^{\alpha}|^2 \d \xi.
\end{align*}

We now prove the Proposition in the case that $1 \leq \alpha <2$.  Taking almost identical steps to those taken in \ref{FTbd}, we find
\begin{align*}
\int_0^\infty \int_{\mathbb R^d} \left(\frac{\nu_{1}^f(x,r)}{r^{\alpha-1}}\right)^2  \frac{\d x \d r}{r} \leq \int_{\mathbb R^d}|\hat f(\xi)|^2 \int_0^\infty
\int_{B_r(0)} |e^{-2\pi i y \cdot \xi}-2\pi i y \cdot \xi \hat \psi(r \xi) - \hat \psi (r\xi)|^2 \frac{ \d y \d r}{r^{d+1+2\alpha}} \d \xi \\
\leq \int_{\mathbb R^d} |\hat f (\xi) |\xi|^\alpha|^2 \int_0^\infty \int_{B_r(0)} |e^{-2 \pi i y\cdot \frac{\xi}{|\xi|}} 
 - 2 \pi i  y \cdot \frac{\xi}{|\xi|} \hat \psi(r) - \hat \psi(r)|^2 \d y \frac{\d r}{r^{d+1 +2\alpha}} \d \xi .
\end{align*}
To estimate the inner integrand when $r \geq 1$, we write
\begin{align}\label{integrandsplit}
|e^{-2 \pi i y\cdot \frac{\xi}{|\xi|}} 
 - 2 \pi i  y \cdot \frac{\xi}{|\xi|} \hat \psi(r) - \hat \psi(r)|
\leq |e^{-2\pi i y \cdot \frac{\xi}{|\xi|}}(1- \hat \psi(r))| + |\hat \psi(r)||e^{-2\pi i y \cdot \frac{\xi}{|\xi|}} - 2 \pi i y \cdot \frac{\xi}{|\xi|} - 1|.
\end{align}
Notice that for such $r$, $|e^{-2\pi i y \cdot \frac{\xi}{|\xi|}}(1- \hat \psi(r))| = |1- \hat \psi(r)| \lesssim 1$.
For the second term above
we use the estimate $|e^{-2\pi i y \cdot \frac{\xi}{|\xi|}} - 2 \pi i y \cdot \frac{\xi}{|\xi|} - 1| \lesssim 1+ r$ and the fact the $\hat \psi$ is Schwartz
$$|\hat \psi(r)||e^{-2\pi i y \cdot \frac{\xi}{|\xi|}} - 2 \pi i y \cdot \frac{\xi}{|\xi|} - 1| \lesssim |\hat \psi(r)|(1+r) \lesssim \frac{1+r}{1+r^2}\leq 1.$$
In the case that $r \leq 1$ we use the same decomposition in \eqref{integrandsplit}.  The first term in this display is bounded on the order of $r^2$ because
the first-order Taylor polynomial of $\hat\psi$ at $0$ is $1$.  For the second term we use that the
first Taylor polynomial of $e^x$ at $0$ is $1+x$ to estimate
$$|\hat \psi(r)||e^{-2\pi i y \cdot \frac{\xi}{|\xi|}} - 2 \pi i y \cdot \frac{\xi}{|\xi|} - 1| \lesssim |y|^2 \leq r^2.$$
So, we have the estimate
\begin{align*}
\int_0^\infty \int_{B_r(0)} |e^{-2 \pi i y\cdot \frac{\xi}{|\xi|}} 
 - 2 \pi i  y \cdot \frac{\xi}{|\xi|} \hat \psi(r) - \hat \psi(r)|^2 \d y \frac{\d r}{r^{d+1 +2\alpha}} \\
\lesssim \int_0^1 r^{d+4}r^{-d-1-2\alpha} \d r + \int_1^\infty \frac{\d r}{r^{1+2\alpha}} < \infty.
\end{align*}
So, we have proved the case where $1 \leq \alpha <2$ as well.
\end{proof}

Before continuing, we need to collect a few facts about the fractional Laplacian.

\begin{proposition}\label{Ialphatosqr}
Suppose $f \in I_{\alpha}(BMO)$, $0 < \alpha<2$.  We have, for all $R>0$ and $z \in \mathbb R^d$, the estimate
\begin{equation}
\int_0^R \int_{B_R(z)} \left(\frac{\overline{\nu}_{\lfloor \alpha \rfloor}^f(x,r)}{r^{\alpha-1}}\right)^2\frac{\d x \d r}{r} \leq C R^{d},
\end{equation}
where the constant $C$ is comparable to $\|D_\alpha f\|_*^2$.
\end{proposition}

\begin{proof}
We split the proof to cover three cases: $1 < \alpha < 2$, $\alpha =1$, and $0<\alpha <1$.    
The case when $\alpha =1$ differs substantially from the other two cases.  The case when $0 < \alpha <1$ is only slightly different from the case
when $1 < \alpha <2$.

We now prove the case when $\alpha = 1$.  Fix $f \in I_{1}(BMO)$, $z \in \mathbb R^d$ and $R>0$.  We localize $f$ to $B_R(z)$ as follows
$$g(x) := \phi\bigg(\frac{x-z}{R}\bigg)\big(f(x)- f(z) + \nabla( f \ast \varphi_R)(z) \cdot(x-z) \big).$$
where $\phi$ is the same as in the statement of Proposition \ref{smoothsio1}.

Notice that $\nu_{1}^g(x,r) = \nu_{1}^f(x,r)$ for $x \in B_R(z)$ and $0 < r \leq R$.  So, we necessarily have
$$\int_0^R \int_{B_R(z)} \left(\overline{\nu}_{1}^f(x,r)\right)^2 \frac{\d x \d r}{r} 
= \int_0^R \int_{B_R(z)} \left(\overline{\nu}_{1}^g(x,r)\right)^2 \frac{\d x \d r}{r}.$$

We apply Proposition \ref{dorronsoro} to say
\begin{align}\label{betasqtoT}
 \int_0^R \int_{B_R(z)} \left(\overline{\nu}_{1}^g(x,r)\right)^2 \frac{\d x \d r}{r} \lesssim  \int_{\mathbb R^d} \left(\hat g (\xi)|\xi|\right)^2 \d \xi
\approx \int_{\mathbb R^d} |\nabla g(x)|^2 \d x
\lesssim \int |I(x)|^2 \d x  + \int |II(x)|^2 \d x.
\end{align}
Where $I(x)$ and $II(x)$ are obtained by estimating $|\nabla g|$ as follows
\begin{multline*}
|\nabla g(x)| \\
\leq \left| R^{-1}(\nabla \phi)\left(\frac{x-z}{R} \right)\big(f(x)-f(z) - \nabla(f\ast \varphi_R)(z) \cdot (x-z) \big) \right|
+\left|\varphi\left( \frac{x-z}{R}\right) \big(\nabla f(x) - \nabla(f \ast \varphi_R)(z) \big) \right| \\
=: I(x) + II(x).
\end{multline*}
We estimate $\int |I(x)|^2 \d x$ first
\begin{align*}
\int_{\mathbb R^d} |I(x)|^2 \d x &\leq R^{-2} \int_{|x-z| \lesssim R} |f(x) -f(z) -\nabla (f \ast \varphi_R)(z) \cdot (x-z)|^2 \d x\\
&= R^{-2} \int_{|x-z| \lesssim R} \left|\int_0^1 \nabla f(tx - (1-t)z) \cdot (x-z) \d t  -\nabla (f \ast \varphi_R)(z) \cdot (x-z)\right|^2 \d x \\
&\lesssim  \int_0^1 \int_{|x-z| \lesssim R} \left|\nabla f(tx+(1-t)z) - \nabla (f \ast \varphi_R)(z)\right|^2 \d x \d t \\
&\lesssim \int_0^1 t^{-d} \int_{|x-z| \lesssim tR} \left|\nabla f(x) - \nabla(f \ast \varphi_{tR})(z)\right|^2 \d x \d t\\
&+\int_0^1  R^d\left|\nabla (f \ast \varphi_R)(z) - \nabla (f \ast \varphi_{tR})(z)\right|^2 \d t = T_1 + T_2.
\end{align*}
We will show that both terms on the last line of the display above are bounded on the order of $R^d\|\nabla f\|_{BMO}^2$. Term $T_1$ can 
be seen to bounded on the order of $R^d\|\nabla f\|_{BMO}^2$ by noting that $\|\varphi(x) - 1_{B_1(0)}\|_{H^1} \lesssim 1$, which, together
with standard John-Nirenberg estimates and $H^1$-$BMO$ duality, yield the estimate
\begin{align*}
 &\int_0^1 t^{-d} \int_{|x-z| \lesssim tR} \left|\nabla f(x) - \nabla(f \ast \varphi_{tR})(z)\right|^2 \d x \d t\\
\lesssim& \int_0^1 t^{-d} \int_{|x-z| \lesssim tR} |\nabla f(x) - \fint_{|y-z| \leq tR}\nabla f(y) \d y|^2 \d x \d t
+R^d|\nabla (f \ast \varphi_{tR})(z) - \nabla(f \ast (tR)^{-d}1_{B_{tR}(0)})(z)|^2  \\
\lesssim& R^d \|\nabla f\|_{*}^2.
\end{align*}
For term $T_2$ we do the following calculation
\begin{multline*}
\left|\nabla (f \ast \varphi_R)(z) - \nabla (f \ast \varphi_{tR})(z)\right| \\
\left| \int (\nabla f)(z-y)\big( \varphi_R(y) - \varphi_{tR}(y)\big) \d y \right| = \left| \int (\nabla f)(z-(Rt)y) \big( t^{d} \varphi(ty) - \varphi(y) \big) \d y \right|\\
\leq \|\nabla f\|_* \|t^{d} \varphi(t\cdot)-\varphi(\cdot)\|_{H^1} \lesssim \|\nabla f\|_*.
\end{multline*}

It remains to estimate $\int |II(x)|^2 \d x$.  This uses the same techniques for the first term, but is even easier.  We omit the details, and conclude the proof
in the case when $\alpha =1$.

Now we consider the cases when $\alpha \neq 1$.  We will use Propositions \ref{smoothsio1} and \ref{smoothsio2} along with an approximation argument.\\

Fix $f \in I_{\alpha}(BMO)$, $1 < \alpha < 2$, $z \in \mathbb R^d$ and $R>0$.  We localize $f$ to $B_R(z)$ as follows
$$g(x) := \phi\bigg(\frac{x-z}{R}\bigg)\big(f(x)-f(z) -\nabla f(z) \cdot(x-z) \big).$$
where $\phi$ is a smooth, positive function with $\phi \equiv 1$ on $\overline{B_2(0)}$, $\phi = 0$ outside of $B_4(0)$, and $0 \leq \phi \leq 1$ everywhere.

Notice that $\nu_{1}^g(x,r) = \nu_{1}^f(x,r)$ for $x \in B$ and $0 < r \leq R$.  So, we necessarily have
$$\int_0^R \int_{B_R(z)} \left(\frac{\overline{\nu}_{1}^f(x,r)}{r^{\alpha-1}}\right)^2 \frac{\d x \d r}{r} 
= \int_0^R \int_{B_R(z)} \left(\frac{\overline{\nu}_{1}^g(x,r)}{r^{\alpha-1}}\right)^2 \frac{\d x \d r}{r}.$$
If we can show that the right hand side is bounded on the order of $R^d$, we will have proved the Proposition
in the case $1 < \alpha <2$.  Let $\{\psi_t\}_{t>0}$ be an approximation of the identity and set
\begin{align*}
f^t(x) &= (f \ast \psi_t)(x),\\
g^t(x) &= \psi\left( \frac{x-z}{R} \right)\big(f^t(x) - f^t(z) - \nabla f^t(z) \cdot (x-z) \big).
\end{align*}
Notice that for $t>0$, $f^t \in C^{\infty}$, and hence we can apply Lemma \ref{smoothsio1} to conclude
\begin{align}\label{rightbound}
\int_{\mathbb R^d} \big||\xi|^\alpha \hat{g^t}(\xi) \big|^2 \d \xi \lesssim \|D_\alpha f^t\|_*^2 R^d.
\end{align}
On the other hand, using arguments which are now familiar to the reader, we have
\begin{align}\label{leftbound}
\int_0^R \int_{B_R(z)} \left(\frac{\overline{\nu}^{f^t}_1(x,r)}{r^{\alpha-1}}\right)^2 \frac{\d x \d r}{r} 
\lesssim \int_{\mathbb R^d} \big(\hat g^t(\xi) |\xi|^\alpha \big)^2 \d \xi.
\end{align}
One can easily check that $\|D_\alpha f^t\|_* \leq C\|D_\alpha f\|_*$ uniformly in $t$.  Using
the H\"older bound on $\nabla f$ we derive the bound
\begin{align}
|\overline{\nu}_1^{f^t-f}(x,r)^2| \lesssim
 \fint_{B_r(x)} \left(\frac{|(f^t(y)-f(y)) -(f^t(x)-f(x)) - \nabla (f^t(x)-f(x))\cdot (y-x)|}{r} \right)^2 \d y
\lesssim t^{2(\alpha-1)}.
\end{align}
Hence,
\begin{multline*}
\int_{\sqrt{t}}^R \int_{B_R(z)} \left(\frac{\overline{\nu}_1^f(x,r)}{r^{\alpha-1}}\right)^2 \frac{\d x \d r}{r}
\lesssim \int_{\sqrt{t}}^R \int_{B_R(z)} \left(\frac{\overline{\nu}_1^{f^t}(x,r)}{r^{\alpha-1}}\right)^2 \frac{\d x \d r}{r} 
+ R^dt^{\alpha-1} \\
\lesssim \int_{\mathbb R^d} \big(\widehat{g^t}(\xi) |\xi|^\alpha \big)^2 \d \xi + R^dt^{\alpha -1} \lesssim \|D_\alpha f\|_*^2R^d + R^dt^{\alpha -1}.
\end{multline*}
Taking $t$ to $0$ yields the Proposition in the case that $1 < \alpha <2$.  As usual, the case $0<\alpha<1$ is almost exactly the same, but easier.  In this case we set
\begin{align*}
f^t(x) &= (f \ast \psi_t)(x),\\
g^t(x) &= \psi\left( \frac{x-z}{R} \right)\big(f_t(x) - f_t(z)\big).
\end{align*}
The only other difference between this case and the case when $1 < \alpha <2$ is that we now have the bound
$$|\overline{\nu}_1^{f^t-f}(x,r)^2| =
 \fint_{B_r(x)} \left(\frac{|(f^t(y)-f(y)) -(f^t(x)-f(x))|}{r} \right)^2 \d y \lesssim t^{2\alpha}.$$
 But one can then check
 $$\int_{\sqrt t}^R \int_{B_R(z)} \left(\frac{\nu_1^{f^t-f}(x,r)}{r^{\alpha-1}} \right)^2 \frac{\d x \d r}{r} \lesssim R^dt^\alpha.$$
 So, the same argument as the one given above works in this case as well.
\end{proof}

We now note that, in the case $1<\alpha<2$, $f \in I_{\alpha}(BMO)$ implies the tempered growth condition
\eqref{tempg1} for $\epsilon > \alpha$.  This is proved in \cite{Str}, see Theorem
\ref{stzthm} below.  So, we have proved one direction of
\ref{mainthm}.

\section{From Square Function Bounds to $I_{\alpha}(BMO)$.}

In this section we prove that the square function bounds \eqref{nu0} and \eqref{nu1} put a function in $I_{\alpha}(BMO)$, $0<\alpha<1$ and $1 \leq \alpha <2$, respectively (note: for the latter we also need 
to assume \eqref{tempg1} for $\epsilon>
\alpha$.)  To do this we will verify a criterion of Strichartz.  The following Theorem is proved
in \cite{Str}.

\begin{theorem}[\cite{Str}]\label{stzthm}
For $0<\alpha<1$, a function $f$ is in $I_{\alpha}(BMO)$ if and only if 
\begin{align}\label{stzcrit1}
\sup_Q \left(\frac{1}{|Q|} \int_Q \int_{Q} \frac{|f(x+y)-f(x)|^2}{|x-y|^{d+ 2\alpha}} \d y \d x \right)^\frac{1}{2} = B<\infty.
\end{align}
where the supremum above is taken over cubes (note: we can take it over balls as well, changing $B$ only by a multiplicative constant depending on dimension.)
For $0 < \alpha <2$, $f \in I_{\alpha}(BMO)$ if and only if $f$ satisfies the following ``tempered growth estimate" for all $\epsilon > \alpha$
\begin{align}\label{tempg}
\int \frac{|f(x)|}{1+ |x|^{d+\epsilon}} \d x <\infty,
\end{align}
and also the estimate
\begin{align}\label{stzcrit2}
\sup_Q \left(\frac{1}{|Q|} \int_Q \int_{Q} \frac{|2f(x) - f(x+y)-f(x - y)|^2}{|x-y|^{d+ 2\alpha}} \d y \d x \right)^\frac{1}{2} = B<\infty.
\end{align}
In both cases, the constant $B$ is comparable to $\|D_{\alpha} f\|_*$.  
\end{theorem}
\begin{remark}
	We note that in \cite{Str} the author refers
	to \ref{tempg} as the ``tempered growth condition", but does not specify that it needs to hold for $\epsilon>\alpha$.  However, if one carefully reads the proof provided in that paper, they will see that $\epsilon>\alpha$
	is the correct range for the statement.
\end{remark}

We need to replace the optimal polynomials of the $\nu$ coefficients with ``approximate Taylor polynomials."  The content of the next proposition
shows that one can replace the infimum in the definition of the $\nu$ coefficients with an approximate Taylor polynomial and still retain the square function estimate.
When $0 < \alpha < 1$, we define 
\begin{align}\label{optimalnu0}
\tilde \nu_{0}^f(x,r) = \left(\fint_{\frac{r}{2} \leq |y| \leq r} \left(\frac{|f(x+y)-f(x)|}{r}\right)^2 
\d y\right)^{\frac{1}{2}}
\end{align}
and, when $1 \leq \alpha < 2$, we define 
\begin{align}\label{optimalnu1}
\tilde \nu_{1}^f(x,r) = \left(\fint_{\frac{r}{2} \leq |y| \leq r} \left(\frac{|f(x+y)-f(x) - (\phi_r \ast \nabla f)(x)\cdot y|}{r}\right)^2 
\d y\right)^{\frac{1}{2}},
\end{align}
where $\phi_r$ is some approximate identity.
\begin{lemma}\label{optimalbetabound}
For $0 <\alpha<2$ and $f$ satisfying
$$\int_0^R \int_{B_R(z)} \left(\frac{\nu_{\lfloor \alpha \rfloor}^f(x,r)}{r^{\alpha-1}}\right)^2 \frac{\d x \d r}{r} \leq C_1R^d$$
for all $R>0$ and $z \in \mathbb R^d$, we have the estimate
$$\int_0^R \int_{B_R(z)} \left(\frac{\tilde \nu_{\lfloor \alpha \rfloor}^f(x,r)}{r^{\alpha-1}}\right)^2 \frac{\d x \d r}{r} \lesssim C_1R^d.$$
\end{lemma}
\begin{proof}

We now prove the Lemma is the case that $0 < \alpha < 1$.  Given $x\in \mathbb R^d$ and $r>0$, let
$c_{(x,r)}$ be the constant minimizing $\nu^f_0(x,r)$.  For $f$ as in the statement of the Lemma, $R>0$, and $z \in \mathbb R^d$, we have
\begin{multline*}
\int_0^R \int_{B_R(z)} \left(\frac{\tilde \nu_{0}^f(x,r)}{r^{\alpha-1}} \right)^2 \frac{\d x \d r}{r} \\
\lesssim \int_0^R \int_{B_R(z)} \fint_{\frac{r}{2} \leq |y| \leq r} \left(\frac{|f(x)-c_{(x,r)}|}{r}\right)^2 \d y r^{1-2\alpha} \d x \d r \\
+\int_0^R \int_{B_R(z)} \fint_{\frac{r}{2} \leq |y| \leq r} \left(\frac{|f(x+y)-c_{(x,r)}|}{r} \right)^2 \d y r^{1-2\alpha} \d x \d r \\
= I + II.
\end{multline*}
Term $II$ is trivially bounded up to a harmless constant
by
$$\int_0^R \int_{B_R(z)}\left(\frac{\nu_{0}^f(x,r)}{r^{\alpha-1}} \right)^2 \frac{\d x \d r}{r} \leq C_1R^d.$$
For term $I$, we change variables in $y$ so the range of integration in the inner integral is over $r/2 \leq |y-x| \leq r$.  Then, we notice that
if $|z-x| \leq R$ and $|x-y| <r$, then $|z-y| \leq 2R$.  So, we can use Fubini to bound the integral from above by
$$\int_0^R \int_{B_{2R}(z)} \fint_{\frac{r}{2} \leq |x-y| \leq r} \left(\frac{|f(x)-c_{(x,r)}|}{r} \right)^2 \d x r^{1-2\alpha} \d y \d r.$$
This integral is clearly bounded on the order of $C_1R^d$.

Now we prove the case when $1 \leq \alpha <2$.  We now let $l_{(x,r)}$ denote the affine function
minimizing $\nu_1^f(x,r)$.  We note that for any affine function $L(x)$ one has $L(x+y) = L(x) + y \cdot \big(\phi_r \ast \nabla L(x)\big)$.  We apply this fact to $l_{(x,r)}$ to obtain
the estimate
\begin{multline*}
\int_0^R \int_{B_{R}(z)} \left(\frac{\tilde \nu_{1}^f(x,r)}{r^{\alpha-1}} \right)^2 \frac{\d x \d r}{r}
\lesssim \int_0^R \int_{B_R(z)} \fint_{r/2 \leq |y| \leq r} \left(\frac{|f(x+y) -l_{(x,r)}(x+y)|}{r} \right)^2 \d y r^{1-2\alpha}\d x \d r \\
+\int_0^R \int_{B_R(z)} \fint_{r/2 \leq |y| \leq r}     \left(\frac{|f(x) -l_{(x,r)}(x)|}{r} \right)^2 \d y r^{1-2\alpha} \d x \d r \\
+ \int_0^R \int_{ B_R(z) }\fint_{r/2 \leq |y| \leq r} \left(\frac{|y \cdot (\phi_r \ast\nabla (f-l_{(x,r)}))(x)|}{r} \right)^2 \d y r^{1-2\alpha}\d x \d r  \\
= I + II + III.
\end{multline*}
We need to show that $I$, $II$, and $III$ are all bounded on the order of $R^d$.  By the definition of $l$, it is clear that
$$I \lesssim \int_0^R \int_{B'_R(z)} \left(\frac{\nu_{1}^f(x,r)}{r^{\alpha-1}}\right)^2 \frac{\d x \d r}{r} \lesssim C_1R^d.$$
Term $II$ is estimated in the exact same way that term $I$ in the case $0 < \alpha <1$ was estimated, which will show that it is bounded on the order of
$C_1R^d$.  For term $III$ we move the gradient in the convolution
off $f-l_{(x,r)}$ and onto $\phi_r$.  This bounds the inner integrand in term $III$ by
$$\left(\frac{y \cdot \frac{1}{r} ((\nabla \phi_r) \ast (f-l_{(x,r)})(x)}{r}  \right)^2 
\lesssim \left(\fint_{B_{r}(x)} \frac{|f(\zeta) - l_{(x,r)}(\zeta)|}{r}\d \zeta \right)^2 \leq \nu_{1}^f(x,r)^2 .$$
Where the last inequality above is an application of H\"older's inequality.  This clearly allows us to bound term $III$ on the order of $C_1R^d$.
\end{proof}

\begin{proposition}\label{othprop}
Suppose $0<\alpha <2$, and $f: \mathbb R^d \rightarrow \mathbb R$ satisfies
\begin{align*}
\int_0^R \int_{B'_R(z)} \left(\frac{\nu^f_{\lfloor \alpha \rfloor} (x,r)}{r^{\alpha-1}} \right)^2 \frac{\d x \d r}{r} \leq C_1 R^d.
\end{align*}
for all $R>0$ and all $z \in \mathbb R^d$.  Additionally,
in the case that $1 \leq \alpha <2$, assume $f$ satisfies
\ref{tempg1} for all $\epsilon>\alpha$.  Then $f \in I_{\alpha}(BMO)$, and $\|D_\alpha f\|_{*}^2 \lesssim C_1$.  
\end{proposition}
\begin{proof}
Fix $R>0$ and $z \in \mathbb R^d$.  By Lemma \ref{optimalbetabound}, our hypothesis implies
$$\int_0^R \int_{B_R(z)}  \left(\frac{\tilde \nu_{\lfloor \alpha \rfloor}^f(x,r)}{r^{\alpha-1}} \right)^2 \frac{\d x \d r}{r}  \lesssim C_1R^d$$
Suppose that $1 \leq \alpha <2$.  The left hand side of the expression above is clearly bounded from below by a constant times
\begin{multline}\label{sqrfrombeta}
\int_0^R \int_{B_R(z)} \fint_{\frac{r}{2} \leq |y| \leq r} \frac{|f(x) - f(x+y) +(\phi_r \ast \nabla f(x))\cdot y
						+f(x) -f(x-y) -(\phi_r \ast \nabla f(x))\cdot y|^2  }{r^{2\alpha-1}} \frac{\d y \d x \d r}{r} \\
=\int_{0}^{R} \int_{B_R(z)}  \fint_{\frac{r}{2} \leq |y| \leq r} \frac{|2f(x) - f(x+y)-f(x-y)|^2}{r^{2\alpha-1}} \frac{\d y \d x \d r }{r}\\
\approx \int_{0}^{R} \int_{B_R(z)} \int_{\frac{r}{2} \leq |y| \leq r} \frac{|2f(x) - f(x+y)-f(x-y)|^2}{|y|^{d+2\alpha}} 
\frac{\d y \d x \d r}{r}\\
\geq \int_{B_{\frac{R}{2}}(z)} \int_{|y| \leq \frac{R}{2}}  \frac{|2f(x) - f(x+y)-f(x-y)|^2}{|y|^{d+2\alpha}} \left(\int_{|y|}^{2|y|} \frac{\d r}{r} \right) 
\d y \d x \\
\approx  \int_{B_{\frac{R}{2}}(z)} \int_{|y| \leq \frac{R}{2}}  \frac{|2f(x) - f(x+y)-f(x-y)|^2}{|y|^{d+2\alpha}} \d y \d x
\end{multline}
The Proposition now follows from Theorem 3.3 of \cite{Str},
stated in Theorem \ref{stzthm} above.

The same Theorem also says that, for $0<\alpha <1$, $f \in I_{\alpha}(BMO)$ if the following estimate holds for all $z \in \mathbb R^d$ and $R>0$
$$\int_{B_R(z)} \int_{|y| \leq R}  \frac{|f(x) - f(x+y)|^2}{|y|^{d+2\alpha}} \d y \d x \leq C_2R^d$$
and, moreover, in this case $\|D_\alpha f\|_*^2 \lesssim C_1$.  So, we prove the Proposition in the case when $0<\alpha <1$ by performing an estimate which is completely analogous to \ref{sqrfrombeta}, but easier.  We omit the details.
\end{proof}
In light of Proposition \ref{othprop}, we have finished the proof of Theorem \ref{mainthm}.

\section{Appendix}
We prove Proposition \ref{Sldual}.  The proof
of this proposition is a modification of the proof
of Proposition 1.1.3 of \cite{G2}.

\begin{proof}[Proof of Proposition \ref{Sldual}]
Let $(\mathcal S_l)^*$ denote the dual of $\mathcal S_l$.  For $u \in \mathcal S'$, let $S(u)$ be the restriction of $u$ to $\mathcal S_l$.  Suppose that 
$u \in \mathcal S'$ is such that $S(u) =0$, i.e. $\langle u,\varphi \rangle = 0$ for all $\varphi \in \mathcal S_l$.  Then, in particular, 
$\langle \hat u, \check \varphi \rangle=0$ for all $\check \varphi \in \mathcal S$ supported away from the origin, and so $u$ must be a polynomial.  
Suppose, by way of contradiction, that the degree of $u$ is $k$ and $k >l$.  Suppose one term of $u$ of degree $k$ is
$$a x_{i_1}^{n_1}...x_{i_j}^{n_j}$$
where $a\neq 0$, and $n_1 + ... + n_j = k$.  Let $\theta: \mathbb R^n \rightarrow \mathbb R$ be smooth, equal to $1$ on $B_1(0)$, and supported in $B_2(0)$.  Set
$$\phi(x) : = \theta(x) \left(a  x_{i_1}^{n_1}...x_{i_j}^{n_j}\right).$$
We note that $\phi \in \mathcal S_l$.  This is true because any derivative of order up to $l$ of $\hat \phi$ must evaluate to $0$ at $x=0$.  We also have
$$\left( \partial_{i_1}^{n_1}...\partial_{i_j}^{n_j} \right) \phi(x)\bigg \vert_{x=0}  \neq 0.$$
Moreover, if $cx_{i'_1}^{n_1'}...x_{i'_{j'}}^{n_j'}$ is any term other term of degree $\leq k$ (a term of any polynomial, not just of $u$), then  
\begin{align*}
\left(\partial_{i'_1}^{n_1'}...\partial_{i_{j'}}^{n_j'} \right)  x_{i_1}^{n_1}...x_{i_j}^{n_j}\bigg \vert_{x=0}  = 0.
\end{align*}
So, we have $\langle \check{u}, \hat \phi \rangle \neq 0$ and $\phi \in \mathcal S_l$, a contradiction.  So, the degree of $u$ must be less than or equal
to $l$.  So, $\ker S \subseteq P_l$.  The other inclusion is obvious from the definitions, hence $\ker S = P_l$. \\

To finish the proof, it is enough so show that the range of $S$ is $(\mathcal S_l)^*$.  Given $u \in (\mathcal S_l)^*$, $u$ is a linear functional 
defined on $\mathcal S_l$, which is a subspace of $\mathcal S$.  So, there is a finite sum of
Schwartz seminorms $Q$ such that, for all $\varphi \in \mathcal S_l$
$$|\langle u, \varphi \rangle| \leq C Q(\varphi).$$
An application of the Hahn-Banach theorem tells us that $u$ has an extension to all of $\mathcal S$, and hence that $S$ is a surjective mapping.
\end{proof}

We need the following Proposition to show
that the distributions $I_{\alpha}(BMO)$ are well
defined.
\begin{proposition}
Given $b \in BMO$, the integral
$$ \int b(x) \varphi(x) \d x$$
is well defined and absolutely convergent for every $\varphi \in \mathcal S_0$.  Moreover, $b$ can be identified with a distribution in
$\mathcal S' / P_0 = (\mathcal S_0)^*$.  In other words, there is a finite sum of Schwartz seminorms $Q_b$ such that for every $\varphi \in \mathcal S_0$
$$|\langle b, \varphi \rangle| := \left|\int b(x) \varphi(x) \d x \right| \leq C Q_b(\varphi).$$
\end{proposition}

\begin{proof}
The fact that the integrals in question are well defined is clear because our test functions have mean value zero.  Additionally, for $b \in BMO$ and $\varphi \in \mathcal S_0$,
$$\int |b(x) \varphi(x)| \d x \lesssim
\int \frac{|b(x)|}{1+ |x|^{d+1}}\d x < \infty,$$  
because $BMO$ functions satisfy \ref{tempg1}
for every $\epsilon>0$.  So the integral in
question is absolutely convergent.

Let $\varphi \in \mathcal S$ be a function with mean value zero.  We use the fact that
$$\int \frac{|b(x) - \langle b \rangle_{B_1(0)}|}{(1+|x|)^{n+1}} \leq c \|b\|_*.$$
For a proof of this fact, see \cite{FS}.  Because $\varphi \in \mathcal S$, $|\varphi(x)| \leq c(\varphi) (1+|x|)^{-n-1}$,
where $c(\varphi)$ is comparable to a finite sum of Schwartz seminorms of $\varphi$ (and these seminorms are the same as $\varphi$ varies.)  
So, we have the estimate
$$ \left| \int b(x) \varphi(x) \d x \right| = \left| \int (b(x) - \langle b \rangle_{B_1(0)}) \varphi(x) \d x \right| \leq c(\varphi)\|b\|_*. $$
By our note above about the nature of $c(\varphi)$, the Proposition follows.
\end{proof}

We recall Definition \ref{ialphabmodef}.\\

\noindent\textbf{Definition 2.5}
	Given $\alpha>0$, $I_{\alpha}(BMO)$ is the subset of $\mathcal S'/P_{\lfloor \alpha \rfloor}$ consisting of distributions $I_{\alpha}(b)$, $b \in BMO$ 
	defined via the following action on $\varphi \in \mathcal S_{\lfloor \alpha \rfloor}$
	$$I_{\alpha}(b)(\varphi) := \int b(x)I_{\alpha}(\varphi)(x) \d x.$$
\\

Let's discuss why these distributions are well defined.  The issues at hand are absolute convergence of the integral and that we need to have $I_{\alpha}(\varphi) \in \mathcal S_0$ when 
$\varphi \in \mathcal S_{\lfloor \alpha \rfloor}$.  For the latter to be true, it is enough to verify that for $\varphi \in \mathcal S_{\lfloor \alpha \rfloor}$
$|\xi|^{-\alpha} \hat \varphi(\xi) \in \mathcal S$ and
$$ \lim_{\xi \rightarrow 0} |\xi|^{-\alpha} \hat\varphi(\xi) = 0.$$
We leave verification of these facts to the reader.
Once we know that $I_{\alpha}(\varphi) \in \mathcal S_0$, the Proposition above
tells us that the integral is well defined and absolutely convergent.

Now we want to verify that $I_{\alpha}(b)$ can be identified with a function when $b \in BMO$.
\begin{proof}[Proof of Proposition \ref{distcoincide}]
  We first note that if $\varphi \in \mathcal S_0$ and $\alpha>0$
$$I_{\alpha}(\varphi)(x) = \int \varphi(y) \left(\frac{1}{|x-y|^{n-\alpha}} - \frac{1}{|x|^{n-\alpha}} \right) \d y.$$
So,
\begin{align*}
I_{\alpha}(b)(\varphi) = \int b(x) \int \varphi(y) \left(\frac{1}{|x-y|^{n-\alpha}} - \frac{1}{|x|^{n-\alpha}} \right) \d y \d x.
\end{align*}
It would be nice if we could use Fubini to switch the order of integration, but it isn't clear that the integral is absolutely convergent.  We replace $b$ with
\begin{align*}
b_k(x) = \begin{cases}
k & b(x) \geq k \\
b(x) &-k \leq b(x) \leq k \\
-k & b(x) \leq -k
\end{cases}
\end{align*}
For the estimates we are about to perform we need that 
$\||\cdot -y|^{\alpha - d} - |\cdot|^{\alpha-d}\|_{H^1} = c|y|^{\alpha}\approx\||\cdot -y|^{\alpha - d} - |\cdot|^{\alpha-d}\|_{L^1}$.  This is 
shown in the proof of Theorem 3.4 in \cite{Str}.  We have
\begin{align*}
\iint \left|b_k(x)\varphi(y) \left(\frac{1}{|x-y|^{d-\alpha}} - \frac{1}{|x|^{d-\alpha}} \right) \right| \d y \d x
 \leq k \int |y|^\alpha|\varphi(y)| \d y <\infty.
\end{align*}
So, we can apply Fubini to conclude that
$$\int b_k(x) \int \varphi(y) \left(\frac{1}{|x-y|^{d-\alpha}} - \frac{1}{|x|^{d-\alpha}} \right) \d y \d x
=\iint b_k(x) \left(\frac{1}{|x-y|^{n-\alpha}} - \frac{1}{|x|^{d-\alpha}} \right) \d x\varphi(y) \d y.$$
So, we have
\begin{multline*}
\left|\int b(x) \int \varphi(y) \left(\frac{1}{|x-y|^{d-\alpha}} - \frac{1}{|x|^{d-\alpha}} \right) \d y \d x
- \iint b(x) \left(\frac{1}{|x-y|^{d-\alpha}} - \frac{1}{|x|^{d-\alpha}} \right) \d x\varphi(y) \d y \right| \\
\leq \left| \int (b(x)-b_k(x)) \int \varphi(y) \left(\frac{1}{|x-y|^{d-\alpha}} - \frac{1}{|x|^{d-\alpha}} \right) \d y \d x\right|
 +\left|\iint (b(x)-b_k(x)) \left(\frac{1}{|x-y|^{n-\alpha}} - \frac{1}{|x|^{d-\alpha}} \right) \d x\varphi(y) \d y \right| \\
\leq \|b-b_k\|_{*} \|I_{\alpha}(\varphi)\|_{H^1}+ \|b-b_k\|_* \int |y|^\alpha |\varphi(y)| \d y.
\end{multline*}
So, taking $k$ to infinity yields 
$$\int b(x) \int \varphi(y) \left(\frac{1}{|x-y|^{d-\alpha}} - \frac{1}{|x|^{d-\alpha}} \right) \d y \d x
= \int \int b(x) \left(\frac{1}{|x-y|^{d-\alpha}} - \frac{1}{|x|^{d-\alpha}} \right) \d x\varphi(y) \d y.$$
Hence, we can associate the distribution $I_{\alpha}(b)$ with the function
$$I_{\alpha}(b)(x) := \int b(y) \left(\frac{1}{|x-y|^{d-\alpha}} -\frac{1}{|y|^{d-\alpha}} \right) \d y.$$
This proves the case when $0<\alpha<1$.  Now
we turn our attention to the case when
$1 \leq \alpha <2$.

Suppose that $f \in I_{\alpha}(BMO)$ for some $\alpha \in [1,2)$.  Let $b \in BMO(\mathbb R^d)$ be such that
$f = I_{\alpha}(b)$.  The partial derivatives of $f$ can be defined distributionally as
\begin{align*}
\partial_i f(\varphi) =-f(\partial_i\varphi),\,\,\, \varphi \in \mathcal S_0.
\end{align*}
Notice that when $\varphi \in \mathcal S_0$, $\partial_i \varphi \in \mathcal S_1$, so this definition makes sense.  Let 
$b \in BMO$ be the function such that
$$f(\varphi) = \langle b,I_{\alpha}(\varphi)\rangle.$$

Let $\mathcal R_i$ denote the $i$\textsuperscript{th} Riesz transform.  Notice that 
$I_{1}(\partial_i \varphi) =  \mathcal R_i(\varphi)$ for any Schwartz function $\varphi$, and some fixed constant $c$ which does not depend on $i$.
Let $H_0^1$ denote the subspace of $H^1$ which consists of finite sums of $H^1$ atoms.  We can define the Riesz transform of $b$ through an action
on $H_0^1$ functions
$$\mathcal R_i b(\varphi) := \langle b, -\mathcal R_i \varphi \rangle,\,\,\, \varphi\in H_0^1.$$
Notice that the mapping sending $\varphi \in \mathcal H_0^1$ to $\langle b,-\mathcal R_i  \varphi \rangle$ is a bounded linear functional on $H_0^1$, and 
hence extends continuously to a bounded linear functional on all of $H^1$ because $H_0^1$ is dense in $H^1$.  Moreover, because $BMO$
is the dual space of $H^1$, there exists some $\tilde b_i \in BMO$ such that
$$\mathcal R_i b(\varphi) = \langle \tilde b_i, \varphi \rangle$$
for all $\varphi \in H_0^1$.  So, we can identify $\mathcal R_i b$ with $\tilde b_i$.  Notice that $\mathcal S_i$ is a subset of $H^1$ for every
$i \geq 0$, and hence we can say that, as a member of $\mathcal S' /P_0$, the distribution $\mathcal R_i b$ coincides with $\tilde b_i$.  \\

For each $\varphi \in \mathcal S_0$, we have
$$\partial_i f(\varphi) = - \langle \tilde b_i,I_{\alpha-1}(\varphi) \rangle = - \langle I_{\alpha - 1}(\tilde b_i), \varphi \rangle.$$
From the results above, we know that $I_{\alpha - 1}(\tilde b_i)$ makes sense as a function.  Hence, $\partial_i f$ coincides with the function
$I_{\alpha-1}(-\tilde b_i) \in I_{\alpha -1}(BMO)$.  We let $\nabla f = (\partial_1 f,...,\partial_d f)$.\\

Consider the function
$$\tilde f(x) = \int_0^1 (\nabla f)(t x) \cdot x \d t.$$
We are going to show that the distribution $f$ coincides with $\tilde f$.  We have, for $\varphi \in \mathcal S_1$,
\begin{align*}
\langle \tilde f ,\varphi \rangle  &=\sum_{i=1}^d \int_{\mathbb R^d} \int_0^1 I_{\alpha-1}(-\tilde b_i)(tx)x_i \d t \varphi(x) \d x\\
&= \sum_{i=1}^d \int_0^1 \int_{\mathbb R^d}  I_{\alpha-1}(- \tilde b_i)(tx) x_i \varphi(x) \d x \d t\\
&=\sum_{i=1}^d \int_0^1 \int_{\mathbb R^d}  I_{\alpha-1}(- \tilde b_i)(x) \frac{x_i}{t} \varphi\left(\frac{x}{t}\right)t^{-d} \d x \d t.
\end{align*}
We justify the interchanging the variables of integration by noting that the integral above converges absolutely, which
can be seen after noting that $|I_{\alpha-1}(-\tilde b_i)(tx) -I_{\alpha-1}(-\tilde b_i)(0)| \lesssim t^{\alpha-1}|x|^{\alpha-1}$ 
and $|x|^{\alpha-1}x_i \varphi(x)$, $x_i\varphi(x)$ are both rapidly decaying.  We now take the Fourier transform in $x$ to obtain
\begin{align*}
\int_0^1 \langle \widehat{I_{\alpha}(b)}(\xi), \xi \cdot (\nabla \check \varphi)(t\xi) \rangle \d t  \\
=\langle \widehat{I_{\alpha}(b)}(\xi), \int_0^1  \xi \cdot (\nabla \check \varphi)(t\xi) \d t  \rangle \\
=\langle I_{\alpha}(b), \varphi \rangle.
\end{align*}
Hence, $\tilde f$ coincides with $I_{\alpha}(b)$.  In the display above, one can justify moving the integration in $t$ inside the brackets
using the fact that the Riemann sums of the integral converge to $\hat\varphi$ in the topology of Schwartz functions and the continuity of $I_{\alpha}(b)$
as a functional.
\end{proof}

\textit{Acknowledgements:} The author thanks Steve Hofmann for several helpful conversations.  He also thanks the anonymous referee for several helpful suggestions which have imporved the manuscript.

\end{document}